\documentclass[a4paper]{amsart}

\usepackage{amsmath, amsthm, amssymb}
\usepackage{amssymb}
\usepackage{graphicx}
\usepackage{subfigure}
\usepackage[export]{adjustbox}

\usepackage{mathrsfs}     

\newcommand\bi{\begin{itemize}}
\newcommand\ei{\end{itemize}}

\def\addsymbol #1: #2#3{$#1$ \> \parbox{5in}{#2 \dotfill \pageref{#3}}\\}

\newtheorem{fact}{Fact}

\newtheorem{defin}{Definition}[section]
          
          \newtheorem{teo}{Theorem}[section]
          
          \newtheorem{con}{Conjecture}
          \newtheorem{cond}{Condition}
          \newtheorem{prop}[teo]{Proposition}
          \newtheorem{lem}{Lemma}[section]
          \newtheorem{rmk}{Remark}
          \newtheorem{cor}{Corollary}[section]
          
          \newcommand{\bfact}{\begin{fact}}
          \newcommand{\efact}{\end{fact}}

          \newcommand{\cA}{A^N}

          \newcommand{\beq}{\begin{equation}}
          \newcommand{\eeq}{\end{equation}}
          \newcommand{\beqn}{\begin{eqnarray}}
          \newcommand{\beqnn}{\begin{eqnarray*}}
          \newcommand{\eeqn}{\end{eqnarray}}
          \newcommand{\eeqnn}{\end{eqnarray*}}
          \newcommand{\bprop}{\begin{prop}}
          \newcommand{\eprop}{\end{prop}}
          
          \newcommand{\bc}{\be\begin{array}{r@{\,}c@{\,}l}}
\newcommand{\ec}{\end{array}\ee}

          \newcommand{\bcor}{\begin{cor}}
          
          \newcommand{\ecor}{\end{cor}}
          \newcommand{\bcon}{\begin{con}}
          \newcommand{\econ}{\end{con}}
          \newcommand{\bcond}{\begin{cond}}
          
          \newcommand{\econd}{\end{cond}}
          \newcommand{\bteo}{\begin{teo}}
          \newcommand{\eteo}{\end{teo}}
          \newcommand{\brm}{\begin{rmk}}
             
          \newcommand{\erm}{\end{rmk}}
          \newcommand{\blem}{\begin{lem}}
          \newcommand{\elem}{\end{lem}}

          \newcommand{\ben}{\begin{enumerate}}
          \newcommand{\een}{\end{enumerate}}
          \newcommand{\bei}{\begin{itemize}}
          \newcommand{\eei}{\end{itemize}}
	 
          \newcommand{\bdf}{\begin{defin}}
          \newcommand{\edf}{\end{defin}}

          \renewcommand{\>}{&>&}
          \renewcommand{\le}{\leq}

          \newcommand{\Z}{{\mathbb Z}}

          \newcommand{\R}{{\mathbb R}}
          
          \newcommand{\E}{{\mathbb E}}
          
          \renewcommand{\P}{{\mathbb P}}

          \newcommand{\N}{{\mathbb N}}

          \newcommand{\chA}{\hat A^N}

          \newcommand{\cV}{{\mathcal V}}

          \newcommand{\D}{{\mathcal D}}

          \newcommand{\hPi}{{\hat \pi}}

\newcommand\bigCI{\mathop{\underline{\raisebox{0pt}[0pt][1pt]{$\;||\;$}}}}

\newcommand{\btt}{\begin{theorem}}
\newcommand{\ett}{\end{theorem}}
\newcommand{\Net}{\mathcal{N}}

\newcommand{\be}{\begin{equation}}
\newcommand{\ee}{\end{equation}}

          \newcommand{\cL}{\lambda}

          \newcommand\sqr{\vcenter{
          \hrule height.1mm
          \hbox{\vrule width.1mm height2.2mm\kern2.18mm\vrule width.1mm}
          \hrule height.1mm}}        

\begin{document}


\begin{center}

\noindent{\large \textbf{The sequential loss of allelic diversity}} 

\noindent{Submitted to appear in the Festschrift in honor of Peter Jagers}

\bigskip

\noindent {\normalsize Guillaume Achaz$^{1, 3}$, Amaury Lambert$^{1,2}$, Emmanuel Schertzer$^{1,2}$}\\
\noindent {\small \it
$^{1}$ Center for Interdisciplinary Research in Biology (CIRB), Coll\`ege de France, CNRS, INSERM, PSL Research University, Paris, France; \\
$^{2}$ Laboratoire de Probabilit\'es, Statistique et Mod\'elisation (LPSM), Sorbonne Universit\'e, CNRS, Paris, France; \\
$^{3}$ Institut de Syst\'ematique, \'Evolution, Biodiversit\'e (ISYEB), MNHN, CNRS, Sorbonne Universit\'e, Paris, France. 
\medskip
}
\end{center}

\paragraph{\bf Abstract.}
This paper gives a new flavor of what Peter Jagers and his co-authors call `the path to extinction'.
In a neutral population with constant size $N$, we assume that each individual at time $0$ carries a distinct type, or allele. We consider the joint dynamics of these $N$ alleles, for example the dynamics of their respective frequencies and more plainly the nonincreasing process counting the number of alleles remaining by time $t$.  We call this process the extinction process. 
We show that in the Moran model, the extinction process is distributed as the process counting (in backward time) the number of common ancestors to the whole population, also known as the block counting process of the $N$-Kingman coalescent. 
Stimulated by this result, we investigate: (1) whether it extends to an identity between the frequencies of blocks in the Kingman coalescent and the frequencies of alleles in the extinction process, both evaluated at jump times; (2) whether it extends to the general case of $\Lambda$-Fleming-Viot processes.




\section{Introduction}

In this note, we are interested in so-called neutral models of population genetics, that is, where all individuals are exchangeable in the face of  death and reproduction \cite{Eth, Ewe}. In such a population, \emph{genetic drift} refers to the randomness of births and deaths and to its effect on the composition of the population. If individuals are given a type, also called \emph{allele}, which is transmitted faithfully to their offspring, one can follow the fluctuations of the numbers of carriers of each allele through time under the sole action of genetic drift \cite{Eth, Ewe}. In the absence of mutation, the number of alleles present in the population can only decrease, and does so exactly at times when the last carrier of a given allele dies. In a population where births and deaths compensate so as to keep its total size constant equal to $N$, the number of alleles decreases sequentially until one single allele remains present, an event called \emph{fixation}. By exchangeability, each allele has the same probability $1/N$ to be the one that fixes.   

Here, we study this process of sequential loss of allelic diversity under the assumption that all individuals initially carry a distinct allele, and its limit as $N\to\infty$. The genealogical model we consider is a classical model of population genetics called the $\Lambda$-Fleming-Viot process \cite{BLG, DK}, whose law is characterized by a finite measure $\Lambda$ on $[0,1]$. When $\Lambda= \Lambda(\{0\})\,\delta_0$, this process is also known as the Moran process \cite{Eth, Ewe}. In this case, all offspring sizes at birth times are a.s. equal to 2. In all other cases, roughly speaking, each offspring size $\xi$ is binomially distributed with parameters $N$ and $p$, where $p$ is drawn in the measure $x^{-2}\Lambda(dx)$ conditional on $\xi \geq 2$.\\
\\
Let us introduce the coalescent (i.e., backwards in time) view of the $\Lambda$-Fleming-Viot process.
The $\Lambda$-Fleming-Viot process is stationary, so one can fix an arbitrary time called present time, and count the number of \emph{ancestors} $\cA_t$ common to the whole population alive at present time, $t$ units of time before the present. We can define more generally the vector $\rho^N(t)$ of frequencies of the descendances at present time of these $\cA_t$ ancestors. The counter $\cA_t$ decreases as $t$ increases and we can study the sequence of its jump times $\left(T^{N,k}\right)_{k=2}^N$, where $T^{N,k}$ is the time at which $\cA$ decreases by crossing or leaving $k$, as time runs backwards (i.e., as $t$ increases). 
 We are also interested in the sequence $\left(\pi_k^N \right)_{k=2}^N$, called the \emph{ancestral block process}, where $\pi_k^N=\rho^N\left(T^{N,k}\right)$ the embedded chain of $\rho^N$ (with possible repeats).
 
 As $N\to\infty$, the processes $(\cA_t;t\geq 0)$ converge in the sense of f.d.d. to some process $(A^\infty_t;t\geq 0)$. It is known that under a condition on the measure $\Lambda$ usually known as CDI (`coming down from infinity') \cite{Sch}, $A^\infty_t<\infty$ for all $t>0$ almost surely. 
  We assume this condition is enforced throughout the paper.
Then the two sequences $\left(T^{N,k}\right)_{k=2}^N$ and $\left(\pi_k^N\right)_{k=2}^N$
converge as $N\to\infty$ in the sense of finite-dimensional distributions. In addition, the limiting sequences $\left(T^{k}\right)_{k\geq 2}$ and $\left(\pi_k^\infty \right)_{k\geq 2}$ are respectively the jump times and the embedded chain of the $\Lambda$-coalescent \cite{P, Sag}.\\
\\
Now let us consider the $\Lambda$-Fleming-Viot process as time runs forward. Fix again an initial time, give a distinct allele to each individual present in the population at this time and define $\chA_t$ as the number of alleles present in the population $t$ units of time later. The topic that gave its title to this note is the study of the sequence of extinction times $\left(\hat T^{N,k}\right)_{k=2}^N$, where $\hat T^{N,k}$ is the time at which $\chA$ decreases by crossing or leaving $k$. Similarly as before, we can also define $\hat\rho^N(t)$ the vector of frequencies of these $\chA_t$ alleles at time $t$ and the sequence $\left( \hat\pi_k^N \right)_{k=2}^N$, called the \emph{haplotype block process}, where $\hat\pi_k^N=\hat\rho^N\left(\hat T^{N,k}\right)$ is the embedded chain of $\hat\rho^N$ (with possible repeats). Here again, under the CDI condition, both sequences converge in the sense of finite-dimensional distributions as $N\to\infty$ (see Proposition \ref{prop:conv-fwd}). \\
\\
We show or recall in Theorem \ref{teo-intro1} two distributional identities: for each fixed $t$, $\rho^N(t)$ and $\hat\rho^N(t)$ have the same law (and thus also $A^N(t)$ and $\hat A^N(t)$) and for each fixed $k\in\{2,\ldots,N\}$, $T^{N,k}$ and  $\hat T^{N,k}$ have the same law. We are interested in generalizing these identities from one-dimensional marginals to processes. Note though that the processes $\rho^N$ and $\hat\rho^N$ cannot be equally distributed since the first one remains constant between jump times while the second one does not. We thus focus our study on jump times and embedded chains.

Our first, striking result (Theorem \ref{teo:1}) is that in the binary case where $\Lambda= \Lambda(\{0\})\,\delta_0$, the  process $(A_t^N;t\geq 0)$ counting the number of ancestors backwards in time  and the process $(\hat A^N_t;t\geq 0)$ counting the number of remaining alleles forwards in time, have the same law (where $N\le\infty$). In particular, the sequences $\left(T^{N,k}\right)$ and  $\left(\hat T^{N,k}\right)$ have the same law.

This result triggers two questions.
\begin{itemize}
\item[(1)] Does this result extend to an identity in law between $\pi_k^\infty$ and $\hat\pi_k^\infty$ in the binary case?
\item[(2)] Does this result extend to an identity in law between $(A_t^\infty;t\geq 0)$ and $(\hat A_t^\infty;t\geq 0)$ in the general case?   
\end{itemize}
The answer to Question (1) is `yes', $\pi_k^\infty$ and $\hat\pi_k^\infty$ are equally distributed in the binary case. Their common distribution is the uniform distribution on the simplex of dimension $k-2$. Therefore, the limiting processes $\rho^\infty$ and $\hat\rho^\infty$ jump at the same rate and have the same one-dimensional distribution at each jump time. We saw that they cannot be equally distributed and actually, even their embedded chains $(\pi_k^\infty)_{k\geq 2}$ and $(\hat\pi_k^\infty)_{k\geq 2}$ are not. This can be seen thanks to Theorem \ref{teo-block-type2} where we prove that $\left( \hat\pi_k^\infty \right)_{k\geq 2}$ is a Markov chain whose transitions are elegantly characterized thanks to a random coupon collection procedure which is obviously distinct from the merging procedure of the Kingman coalescent.

The answer to Question (2) is `no' in general, simply because the jumps of $(A_t^\infty;t\geq 0)$ can take arbitrary values in $\Z^-$, while those of $(\hat A_t^\infty;t\geq 0)$ are in most known cases of measures $\Lambda$, all equal to $-1$ a.s. This last property was conjectured by Labb\'e to always hold \cite{L}. Our Proposition \ref{extinction-polya} gives a criterion equivalent to this property based on urn models, that might be helpful in the future to prove this conjecture.

In the next two sections, we state our main results, first in the binary case and then in the general case. The remainder of the paper will be devoted to proving these results.

From now on, we will denote the set $\{1,\ldots,N\}$ by $[N]$ and the set $\N\cup\{0\}$ by $\N_0$. 

\section{The binary case}

In this section, we limit ourselves to the binary case, i.e., the case when $\Lambda=\Lambda(\{0\})\,\delta_0$, corresponding to the Moran process in forward time, and to the Kingman coalescent in backward time. We introduce our main tools and results in this section, and will generalize them in the next one.

\subsection{The Moran model and the $N$-Kingman coalescent.} 
Consider a population of constant size $N$ and assume that at time $t=0$, every individual is assigned a distinct type, say with values in $[N]$.
We assume that the population evolves according to classical Moran dynamics. 
Every ordered pair $(a,b)$ of individuals is equipped with an independent Poisson clock of rate $1$. When a clock associated to a pair $(a,b)$ rings, $a$ gives birth to a new individual inheriting her type and $b$ dies simultaneously. 
Starting from the present, one can trace backwards in time the genealogy of the population.
It is well known \cite{K} that this genealogy is described in terms of the 
$N$-Kingman coalescent.  

\subsection{Block counting and extinction processes}
In the framework introduced previously, we define the \emph{extinction process} as
\be\label{def-extinction} \chA_t  = \# \{\mbox{alleles present at time $t$} \}  \qquad t\geq 0,
\ee
and
\be
\label{def-extinction-time}
  \hat T^{N,k}   = \sup\{t>0 \ : \ \chA_t \ \geq \ k \}   \qquad k\in\{2,\ldots, N\},
 \ee
so that $\hat T^{N,2}$ is the fixation time of the population, i.e., the first time when one of the initial alleles has invaded the whole population.
In the Kingman coalescent, 
define the \emph{block counting process} as 
\be\label{def-block} \cA_t \ = \ \#  \{\mbox{ancestors at time $t$}\}\qquad t\geq 0,
\ee
and
\be
\label{def-block-time}
  T^{N,k}  \ = \ \sup\{t>0 \ : \ \cA_t \ \geq \ k \}\qquad k\in\{2,\ldots, N\},
   \ee
so that $T^{N,k}$ is the first time that the number of blocks goes from $k$ to $k-1$. It is well known that
the block counting process $\cA_t$ is a pure-death process starting from $N$
with transition rate $k(k-1)/2$ from level $k$ to $k-1$,  $k\in\{2,\ldots,N\}$.

The next result states that the sequential decrease of ancestral lineages in the Kingman coalescent has the same law as
the sequential loss of allelic diversity forwards in time.

\begin{teo}\label{teo:1}
For every $N\in\N$, the extinction process $\chA$ and the block counting process $\cA$
are identical in law.
\end{teo}
This theorem is proved in Section \ref{sect:prrofth1}.

\subsection{The ancestral and haplotype block processes.}
We define  the set of finite mass partitions of $[0,1]$ as the set
\[E  \ = \left\{ u \in \bigcup_{l\in\N} [0,1]^l\ : \ \sum_{i=1}^{\cL(u)} u_i \ = \ 1 \right\}\]
where the $u_i$'s denote the coordinates of the vector $u$ and $\cL(u)$ its length.

Let us define $\pi^{N}_{k}\in E$ as the vector of frequencies of the blocks 
in the $N$-Kingman coalescent at time $T^{N,k}$. Note that $\pi^{N}_{k}$ has exactly $k-1$ coordinates. We call the process $\left( \pi^N_{k}; \ k\in\{2,\ldots,N\} \right)$ the {\it ancestral block process}.
In the next theorem, $\D_{k}$ will refer to the random mass partition induced by a Dirichlet random variable with parameter $\underbrace{(1,\ldots,1)}_{k \ \mbox{times}}$, that is the uniform distribution on the simplex of dimension $k-2$.
We recall the next result.

\begin{teo}[The ancestral block process, Kingman \cite{K}]\label{teo-block-ancestral}
As $N\to\infty$
\[ \left( \pi_k^N \right)_{k=2}^N \Longrightarrow  \left( \pi_k^\infty \right)_{k=2}^\infty,  \ \ \mbox{in the sense of f.d.d.}\]
where the limiting vector is uniquely determined by the property that for every $K\in\N$,
the chain $\left( \pi^\infty_{K-k}; \ k\in\{0, \ldots, K-2\} \right)$ is a Markov chain such that
\begin{enumerate}
\item The initial distribution $\pi^\infty_K$ is distributed as $\D_{K-1}$.
\item The mass partition at time $k$
is obtained from the mass partition at time $k-1$ by merging two blocks chosen uniformly at random among all possible pairs of blocks.
\end{enumerate}
\end{teo}

\begin{rmk}\label{rmk1}
This sequence was considered by Bertoin and Goldschmidt, case $k=1$ in \cite{BG}. In particular, the chain $\left( \pi^\infty_{k}; \ k\in\{2, \ldots, K\} \right)$ is also Markovian: The mass partition at time $k$
is obtained from the mass partition at time $k-1$ by fragmenting uniformly one block chosen as a size-biased pick from the mass partition at time $k-1$.
\end{rmk}

Let us now draw a connection with the forward dynamics. In the Moran model, we can partition the population into blocks of individuals carrying the same type. Analogously to the ancestral block process, we define $\hat \pi^{N}_{k}$ to be the mass partition induced by the types at time $\hat T^{N,k}$. Note that $\hat\pi^{N}_{k}$ has at most $k-1$ coordinates. 
We call the process $\left( \hat \pi^N_{k}; \ k\in\{2, \ldots, N\} \right)$ the {\it  haplotype block process}.

In order to describe our next result, we will need some further notation. 
First, we say that a mass partition is {\it non-degenerate} if all the coordinates are strictly positive. For a given 
mass partition $m$, we can think of $m$ as a discrete probability distribution on $[\cL(m)]$. 
Then let $(Y_n; n\geq 1)$ be an infinite sequence 
of i.i.d. random variables with distribution $m$ and define 
\[T \ = \ \inf\{n\geq1 : \ \forall i\in [\cL(m)], \exists k\in[n], \  Y_k = i    \}, \]
which is a.s. finite if $m$ is non-degenerate and taken otherwise equal to $+\infty$ by convention otherwise.
In the context of the coupon collector problem, $T$ is the time needed to collect all the coupons $1,\ldots, \cL(m)$
where the probability to draw coupon $i$ is equal to $m_i$. 
In light of this observation, for any non-degenerate $m$ we define
\[B_i^{(m)} \ = \ \sum_{k=1}^{T-1} {1}_{Y_k=i}\qquad i\in[\cL(m)], \]
as well as $J^{(m)} = Y_T$, so that $J^{(m)}$ is the unique $i$ such that $B_i^{(m)}=0$.
Then we refer to the random vector
\[B^{(m)} \ = \ \left(B_i^{(m)}; i\in[\cL(m)], i\not= J^{(m)}\right) \]
 as the {\it random coupon collection} 
associated to $m$. 

\begin{teo}[The haplotype block process]\label{teo-block-type}
As $N\to\infty$,
\[ \left( \hat \pi_k^N \right)_{k=2}^n \Longrightarrow  \left( \hat \pi_k^\infty \right)_{k=2}^\infty \mbox{in the sense of f.d.d.} \]
where the limiting vector is uniquely determined by the property that for every $K\in\N$,
the chain $\left( \hat  \pi^\infty_{K-k}; \ k\in\{0, \ldots, K-2\} \right)$ is a Markov chain such that
\begin{enumerate}
\item The initial distribution $\hat \pi^\infty_K$ is distributed as $\D_{K-1}$.
\item The mass partition at time $i$
is obtained from the mass partition at time $i-1$ as follows. 

Conditional on  $\hat \pi^\infty_{K-i}$, generate $(X_i)$ distributed as 
the random coupon collection associated to the mass partition $\hat \pi^\infty_{K-i}$. 

Conditional on $X_i$, the block process at step $i+1$ (i.e., $\hat\pi^\infty_{K-(i+1)}$)
is distributed as a Dirichlet random variable with parameters $(X_i)$.
\end{enumerate}
\end{teo}
The preceding result is a special case of Theorem \ref{teo-block-type2}, which will be proved in Section \ref{sect:proof-bt}.

\begin{rmk}
The vector $\hat\pi^\infty_k$ gives the frequencies of all the alleles present at time $\hat T^{N,k}$. For each of these alleles, the subpopulation carrying this allele at time $\hat T^{N,k}$ has a certain number of ancestors living at time $\hat T^{N,k+1}$. This number is precisely the number of balls (or coupons) corresponding to this allele in the random coupon collection procedure mentioned in Theorem \ref{teo-block-type}.
\end{rmk}

\begin{rmk} From Theorems \ref{teo-block-type} and \ref{teo-block-ancestral}, we observe that the two mass partitions $\pi^\infty$ and $\hat \pi^\infty$ have the same one-dimensional marginal distributions but certainly have different probability transitions. In relation to Remark \ref{rmk1} and the work by Bertoin and Goldschmidt \cite{BG}, an open question remains as to whether $\left( \hat  \pi^\infty_{k}; \ k\in\{2, \ldots, K\} \right)$ is Markovian.
\end{rmk}

\section{The general case}

\subsection{$\Lambda$-coalescent and $\Lambda$-Fleming-Viot (FV) process}

We now turn to more general exchangeable population models.
Let $\Lambda$ be a finite measure on $[0,1]$. 
Consider a population of constant size $N$ evolving according to the following dynamics. 
As before, at time $t=0$, we assign a distinct allele to each individual in the population.  Each $k$-tuple ($2 \leq k\leq N$)
of individuals is equipped with a Poisson clock with rate
\[\lambda^N(k) \ := \ \int_0^1 x^{k-2}(1-x)^{N-k} \Lambda(dx).  \]
When the clock associated with a group of $k$ individuals rings, one of them is picked uniformly at random among them, gives birth to $k-1$ new individuals inheriting her type, while the $k-1$ others die simultaneously.
We call such a process a $(N,\Lambda)$-Fleming-Viot (FV) process. The case $\Lambda = \Lambda(\{0\})\,\delta_0$ corresponds to the classical Moran model presented above.

It is well known that the genealogy of the population can be described in terms of the $(N,\Lambda)$-coalescent, i.e.,
the projection of the $\Lambda$-coalescent on $[N]$ \cite{BLG,DK}.

\subsection{The block counting and extinction processes, continued} For the $(N,\Lambda)$-FV process and $(N,\Lambda)$-coalescent,
we define the extinction process $\chA$ and the block counting process $\cA$ as in (\ref{def-extinction}) and (\ref{def-block}) respectively, along with their jump times $\hat T^{N,k}$ and $T^{N,k}$ as in (\ref{def-extinction-time}) and (\ref{def-block-time}), respectively. The next result gives some identities between one-dimensional marginals
of the block counting and the extinction processes in the general case.

\begin{teo}\label{teo-intro1} Let $N\in {\mathbb N}$.
\begin{enumerate}
\item[(1)]  For every fixed $ t>0$, $\cA_t$ and $\chA_t$ have the same law.
\item[(2)](H\'enard \cite{H}) For every fixed $k\in\{2,\ldots,N\}$,    $T^{N,k}$ and  $\hat T^{N,k}$ have the same law.
\end{enumerate}
\end{teo}
\begin{proof} We only prove (1).
Consider the population at time $t>0$. By stationarity of the Fleming-Viot process, the number of ancestors $a^N_t$ at time 0 of this population is equal in distribution to $\cA_t$. Since each of these $a_t^N$ ancestors carries a distinct allele, the number of alleles present in the population at time $t$ is $a^N_t$, that is $\chA_t = a^N_t$. This shows that $\chA_t$ has the same law as $\cA_t$. The same argument actually shows that the identity also holds for the frequencies of descendances (see Introduction). 
\end{proof}
In the light of Theorem \ref{teo-intro1}, 
it is natural to conjecture that the block counting process and the extinction process are identical in law for any measure $\Lambda$. 
We shall now see that this is not the case in general. In order to see that, we  will let the size of the population go to $\infty$.

In the following, we will be interested in the particular case where for every $k\geq2$, $(T^{N,k}, \hat T^{N,k}; N\geq0)$ is a tight sequence of random variables as $N\to\infty$.
This corresponds to the case where the $\Lambda$-coalescent {\it comes down from infinity} (CDI) which is equivalent to the following condition \cite{Sch}
\begin{equation}\label{eq:cond}
\int^\infty_1 \frac{1}{\psi(u)} du <\infty, \  \ \mbox{where} \ \psi(u) \ = \ \Lambda(\{0\}) u^2 \ + \ \int_{(0,1)} (e^{-xu} - 1 + xu) \frac{\Lambda(dx)}{x^2}. 
\end{equation}
\begin{prop} 
\label{prop:conv-fwd}

Under condition \eqref{eq:cond}, there exist two processes $A^\infty$ and ${\hat A}^\infty$ coming down from infinity such that
as $N\to\infty$, $\cA$ and $\chA$ converge in the sense of f.d.d. to $A^\infty$ and ${\hat A}^\infty$ respectively. In addition, the sequences $(T^{N,k})_k$ and $(\hat T^{N,k})_k$ converge in the sense of f.d.d. to the corresponding jump times of $A^\infty$ and ${\hat A}^\infty$, respectively.
\end{prop}

This proposition will be proved in Section \ref{sect:fixation-coalescence-lines} where the previous statement is reformulated in terms of the Look Down process.

\begin{teo}[Labb\'e \cite{L}, Theorem 1.6]\label{teo:labbe} If $\Lambda$ is such that {\it $\psi$ is regularly varying at $\infty$ with index in $(1,2]$}
then there are no simultaneous extinction events in the $\Lambda$-Fleming Viot process, i.e.,
\be\label{eq:simu-ext}{\mathbb P}\left( \exists t > 0 \ : \ |{\hat A}^{\infty}_{t} - {\hat A}^{\infty}_{t^-}| \geq 2 \right) \ = \ 0.\ee
\end{teo}

\begin{con}[Labb\'e \cite{L}]\label{conj:labbe} (\ref{eq:simu-ext})  holds for any $\Lambda$--FV in the CDI class (i.e., satisfying \eqref{eq:cond}).
\end{con}
Now note that when $\Lambda \neq \Lambda(\{0\}) \,\delta_0$, it is obvious that 
\[{\mathbb P}\left( \exists t > 0 \ : \ | A^{\infty}_{t} - A^{\infty}_{t^-}| \geq 2 \right) \ = \ 1,\]
since $(A^{\infty}_{t},t>0)$ is the number of blocks in the $\Lambda$-coalescent, which by assumption has multiple mergers.
Thus, under the assumptions of the previous theorem and whenever $\Lambda\neq \Lambda(\{0\})\,\delta_0$, 
the block counting process and the extinction process cannot have the same law. Thus, even if 
Theorem \ref{teo-intro1} suggests a strong relation between the block counting and the extinction processes, those two processes are in general very different
when considering coalescent processes with multiple mergers.

\subsection{$\Lambda$-urn. Haplotype block process, continued.} In this section, we
present an extension of Theorem \ref{teo-block-type} for $\Lambda$-coalescent processes in the CDI class.
Define $Y^n$ the r.v. valued in $[n-1]$ with the following distribution
\begin{equation} \label{eq:yn}
  \P\left(Y^n=k-1\right) \ = \ \frac{\int_{[0,1]} {n \choose k} x^{k-2} (1-x)^{n-k}   \Lambda(dx) }{ \int_{[0,1]}  (1-n(1-x)^{n-1} - (1-x)^n) \Lambda(dx) }\qquad k\in\{2,\ldots, n\}.
\end{equation}
With this r.v. at hand, we now define an urn model as follows. At step $n$, the urn contains $U_n$ colored balls, and we let $B(n)$ denote the vector containing the numbers of balls of each color.
Take $B(0) =  \ (B_1(0),\ldots,B_c(0))\in \N^c$ as the initial configuration of the urn, where $c$
is the initial number of colors in the urn.
We now update the configuration of the urn as follows.
Conditional on $B(n)$, the configuration of the urn at time $n+1$ is obtained 
by drawing a ball uniformly at random from the urn and by replacing the ball together with $y^n$
balls of the same color, where $y^n$ is an independent random variable 
distributed as $Y^{U_n}$. We call this urn model a \emph{$\Lambda$-urn}
with initial condition $B(0)$.

\begin{prop}\label{ref:conv}
Let $(B(n); n\geq0)$ be a $\Lambda$-urn with initial condition ${\mathcal B}$ and let $c=\cL({\mathcal B})$ be the number of colors. Let $m_n^{\mathcal B}$
be the mass partition induced by the color frequencies 
$$
\left( \frac{B_i(n)}{\sum_{j=1}^c B_{j}(n)} \right)_{i\in [c]}.
$$
Then there exists a random mass partition  $m_\infty^{\mathcal B}$ on $[c]$ such that 
 \[ m_n^{\mathcal B}  \ \ \rightarrow \  m_\infty^{\mathcal B}  \ \ \mbox{ \ a.s. \ as $n\to\infty$}.\]

\end{prop}
\begin{proof}
As in a standard Poly\`a urn, the proof follows by noting that each color frequency
is a non-negative martingale.
\end{proof}

\begin{rmk}
Note that the limiting  $m_\infty^{\mathcal B}$ is potentially {\it degenerate}, i.e., 
some of the coordinates might be equal to $0$.
\end{rmk}
In the following, we define $\D^{\Lambda}_k$ the limiting random variable $m_\infty^{\mathcal B}$
when we start from the initial condition ${\mathcal B} \ = \ \underbrace{(1,\ldots,1)}_{k \ \mbox{times}}$.
In the case $\Lambda = \Lambda(\{0\})\,\delta_0$,  the $\Lambda$-urn coincides with the classical Poly\`a urn, so that $\D^{\Lambda}_k$ is the Dirichlet distribution $\D_k$.
As a consequence, Theorem \ref{teo-block-type} is a special case of the next result.

\begin{teo}\label{teo-block-type2}
Assume that $\Lambda$ belongs to the CDI class and that there are no simultaneous extinction events in ${\hat A}^\infty$ a.s..
Then 
\begin{enumerate}
\item $\D_{K}^\Lambda$ is non-degenerate a.s.
\item  As $N\to\infty$
\[ \left( \hat \pi_k^N \right)_{k=2}^n \longrightarrow  \left( \hat \pi_k^\infty \right)_{k=2}^\infty \mbox{in the sense of f.d.d.} \]
where the limiting vector is uniquely determined by the property that for every $K\in\N$,
the discrete process $\left( \hat  \pi^\infty_{K-k}; \ k\in\{0,\ldots,K-2\} \right)$ is a Markov process such that
\begin{enumerate}
\item The initial distribution $\hat \pi^\infty_K$ is distributed as $\D_{K-1}^\Lambda$.
\item The configuration at time $i$
is obtained from the configuration at time $i-1$ as follows. 

Conditional on  $\hat \pi^\infty_{K-i}$, let $X_i$ be 
the coupon collection associated to $\hat \pi^\infty_{K-i}$. 

Conditional on $X_i$, the block process at step $i+1$ (i.e., $\hPi^\infty_{K-(i+1)}$)
is distributed as $m_\infty^{X_i}$.
\end{enumerate}

\end{enumerate}
\end{teo}
The preceding result is proved in Section \ref{sect:proof-bt}.

\begin{rmk} In the binary case, we have characterized the limits of the vectors of frequencies of descendances at jump times of both the common ancestors ($\pi^\infty$, Theorem \ref{teo-block-ancestral}) and the surviving alleles ($\hat \pi^\infty$, Theorem \ref{teo-block-type}). Note that in the general case, we only do so for the latter (Theorem \ref{teo-block-type2}).
\end{rmk}

Finally, we end with a result that may shed some light on Conjecture \ref{conj:labbe}.
\begin{prop}\label{extinction-polya} The  following two statements are equivalent
\begin{itemize}
\item[(i)] There are no simultaneous extinctions in ${\hat A}^\infty$.
\item[(ii)] For every $k\in\N$, the limiting mass partition  ${\mathcal D}^\Lambda_{k}$ is non-degenerate.
\end{itemize}
\end{prop}

This proposition will be proved in Section \ref{sect:proof-bt} together with Theorem \ref{teo-block-type2}.

 In the next section, 
  we recall  the construction of the $(N,\Lambda)$--FV process
from the Look-Down (LD) process of Donnelly and Kurtz \cite{DK}. This construction will be the cornerstone of our comparison between extinction and coalescence.
In Section \ref{sect:prrofth1}, we specialize in the case where $\Lambda=\Lambda(\{0\})\,\delta_0$ and we give a proof of Theorem \ref{teo:1}. In Section \ref{general-urn}, we present a general urn model
(encompassing the $\Lambda$-urn models) and give a particle representation {\it \`a la} Donnelly and Kurtz \cite{DK}. We then use this representation 
to prove Theorem \ref{teo-block-type2} (and so Theorem \ref{teo-block-type} as a special case) in Section \ref{sect:proof-bt}.

\section{The Look Down (LD) construction}\label{sect:fixation-coalescence-lines}

Our results are based on the Look-Down (LD) process as defined by Donnelly and Kurtz \cite{DK}.
In the sequel, we will work directly with the LD process, without referring to $\Lambda$-FV processes anymore. We will abuse notation, and use the same symbols 
for the coalescence, extinction times etc. as the ones used in the introduction.  This abuse of notation will be based on the fact that all those quantities in the LD and their analog for the $(N,\Lambda)$-FV are known to be identical in distribution \cite{DK, H}.
At time $t=0$, each level is assigned a distinct allele, where here alleles are i.i.d. uniform in $[0,1]$. Now we let $\nu$ denote a Poisson point measure in $\R\times \mathcal{P}(\N)$ with intensity measure $dt\otimes dL$, where $L$ is defined by 
\begin{equation}
\label{eqn:def-LD}
L = \int_{(0,1]}x^{-2}\Lambda(dx) P_x + \Lambda(\{0\}) \sum_{i<j}Q_{ij} ,
\end{equation}
with $P_x$ the law of the set of 1's in a sequence of i.i.d. Bernoulli r.v. with parameter $x$ and $Q_{ij}$ the Dirac mass at $\{i,j\}$. Each atom $(t, A)$ of $\nu$ is called an \emph{LD-event}. At each LD-event $(t,A)$, all levels $i\in A$ inherit the allele carried by level $\min A$. In terms of the genealogy, the lineage present at level $\min(A)$ splits at time $t$ and gives birth simultaneously to new lineages whose locations are the levels labelled by $A$. All other lineages are untouched but are shifted upwards so that lineages do not cross (see Fig \ref{taut}). More rigorously, if $(t,A)$ is an LD-event, the lineage present at level $i\not\in A$ at $t-$ jumps to level $i+\Delta$ at time $t$, where
$$
\Delta = \#  ((A\setminus \min A)\cap [i]).
$$
Note that $\Delta\in \{0,1\}$ a.s. when $\Lambda$ only charges 0.

\medskip

{\bf Fixation lines.} Let $k\in \N$ and let $U_k$ denote the type carried by level $k$ at time $0$. Then define $(L^k_t; t\geq 0)$ the $k^{th}$ {\it fixation line} \cite{DDSJ, PW} as the right continuous jump process equal at time $t$ to the minimum of all levels carrying type $U_k$ at time $t$. In particular, $L^k_0=k$. The fixation line can be constructed  forwards in time as follows. Conditional on $L^k_t=j$, set
$$
v:=\inf\{s > t \ : \ (s, A)  \mbox{ is a LD event with $\#  (A\cap [j])\geq 2$}\}.
$$
Then $L^k_s=j$ on $[t,v)$ and $L^k_v=j+\Delta$ where $\Delta = \#  ((A\setminus \min A)\cap [j])$.
Define
\be\label{def:sigmaLD}\hat T^{N,k}\ = \ \inf\{t > 0 \ : \  L^k_t \geq N+1 \}, \ \ \hat T^{k}\ = \ \sup\{t > 0 \ : \  L^k_t <\infty \}  \ee
Since $L_t^k$ is the minimum of all levels carrying type $U_k$, $\hat T^{N,k}$ is exactly the time when $U_k$ disappears from the population occupying the first $N$ levels. Recall $\chA_t$ denotes the number of alleles present in this population at time $t$. Observe that by construction, $L_t^{k-1}< L_t^k$ for all $t$ a.s., so that 
\begin{itemize}
\item
For all $t<\hat T^{N,k}$, $\chA_t \geq k$, 
\item
For all $t\geq\hat T^{N,k}$, $\chA_t < k$,  
\end{itemize}
so $\hat T^{N,k}= \sup\{t>0  :  \chA_t  \geq k \} $ exactly as defined in \eqref{def-extinction-time}. Also, $\hPi^{N,k}$ is the vector of frequencies of each of the $N$ initial alleles at time $\hat T^{N,k}$. Note that the length of $\hPi^{N,k}$ is equal to or smaller than $k-1$ and that its components take values in $\{0, \frac 1N, \frac 2N, \ldots, 1\}$.

\begin{proof}[Proof of Proposition \ref{prop:conv-fwd}]
We only prove the result for the forward process $\hat A^{N}$. The backward process $A^N$  
can be handled in a similar way.

We assume that $(\hat T^{N,k};k\geq 1)_N$ and $\hat T^k$ are coupled by the Look-Down. It is clear that for every $k$, the sequence $\{\hat T^{N,k}; N\geq k\}$ is non-decreasing and converges to $\hat T^k$ a.s.
Further, since $\hat T^{N,k} = T^{N,k}$ in law (by Theorem \ref{teo-intro1}(2)), and since $T^{N,k}$ converges in distribution to the $k^{th}$ coalescence time $T^k$ of a $\Lambda$-coalescent
coming down from infinity (under Assumption (\ref{prop:conv-fwd})), it follows that $\hat T^k$ is identical in law with $T^k$ and so is finite a.s.. 
Let us now consider the process
\[\hat A^\infty_t \ = \ \sup\{k\geq1 \ : \ \hat T^k >t\}\]
so that the jump times of $\hat A^N$ converge a.s. to the ones of $\hat A^\infty$.
For every deterministic $t$, it is not hard to see that $\P\left(\hat A^\infty_{t} - \hat A^\infty_{t^-}>0\right) \ =  \ 0$ (i.e, $\hat A^\infty$
does have any fixed point of discontinuity), and thus, 
for every $t\geq 0$ the sequence
$\{\hat A^N_t; N\geq 1\}$ converges to $\hat A^\infty_t$ a.s.. This yields that $\hat A^N$ converges to $\hat A^\infty$ in the sense of f.d.d..
\end{proof}

From now on, we assume that the $\Lambda$-FV dynamics are in the CDI class, that is, condition \eqref{eq:cond} is enforced. We will also  that there are no simultaneous extinctions, i.e.,
\[\mbox{for every $k\geq 2$, \ $\hat T^k <\infty$  and $\hat T^{k+1} < \hat T^{k}$ a.s.}\]
According to Theorem  \ref{teo:labbe}, those two assumptions will be enforced as soon as $\psi$ is regularly varying at $\infty$
with index in $(1,2]$. In particular, this encompasses the binary case $\Lambda = \Lambda(\{0\})\,\delta_0$.

\section{Proof of Theorem \ref{teo:1}}\label{sect:prrofth1}
In this section we consider the extinction process defined from the LD process when $\Lambda = \Lambda(\{0\})\,\delta_0$. Recall that in this case, there are no simultaneous extinction events so that all jumps of $\chA$ are equal to $-1$ a.s., just as for the block-counting process $\cA$ of the Kingman coalescent. 
In order to prove Theorem \ref{teo:1}, we therefore only need to prove the equality in law
\begin{equation}\label{ee}
(T^{N,k}; \ \  2 \leq k\leq N) \ \stackrel{\mathscr L}= \ (\hat T^{N,k};\ \  2 \leq k\leq N),
\end{equation}
where here $T^{N,k}$ is the time at which the block counting process goes from $k$
to $k-1$.
First, we note that the definition (\ref{def:sigmaLD}) immediately implies that 
$$
\hat T^{N,k} \ = \ e_{k} \ + \cdots  \ + e_N
$$
where $e_j$ is the time needed for the $k^{th}$ fixation line to make a transition from level $j$ to $j+1$. By definition of the LD process,
the r.v.'s $e_j$'s are independent and $e_j$ follows the exponential distribution with parameter $\binom{j}{2}$. On the other hand,
by considering the successive coalescence times in the $N$---Kingman coalescent,
$T^{N,k}$ can be decomposed analogously:
$$
T^{N,k} \ = \ e_N' + \cdots +e_{k}'
$$
where $e_j'$ is the time needed for the coalescent to go
from $j$ blocks to $j-1$ block(s), that is $(e_j'; k\leq j \leq N)$
is identical in law to $(e_j; k\leq j \leq N)$. 

The previous argument shows that the one-dimensional marginals of the two vectors in (\ref{ee})
coincide.
In order to prove (\ref{ee}), it remains  to show that 
$\left(\hat T^{N,k-1}-\hat T^{N,k}; 3\leq k \leq N\right)$ is a sequence of independent random variables. 
By a simple induction, this boils down to proving that 
\begin{equation}\label{e2}
\forall  k\leq j \leq N, \   \hat T^{N,k-1}-\hat T^{N,k} \mbox{ is independent of $(\hat T^{N,i}; \ k \leq i \leq  N)$}. 
\end{equation}
Define 
$$
R^{N,k}:= L_{\hat T^{N,k}}^{k-1},
$$
i.e. $R^{N,k}$ is the position of the $(k-1)^{th}$ fixation line at extinction time $\hat T^{N,k}$ (see Fig \ref{taut}).
The crucial observation is that  $\hat T^{N,k-1}-\hat T^{N,k}$ only depends on
\begin{enumerate}
\item[(a)] $R^{N,k}$. 
\item[(b)]  the LD events after time $\hat T^{N,k}$.
\end{enumerate}

\begin{figure}
\includegraphics[scale=0.3]{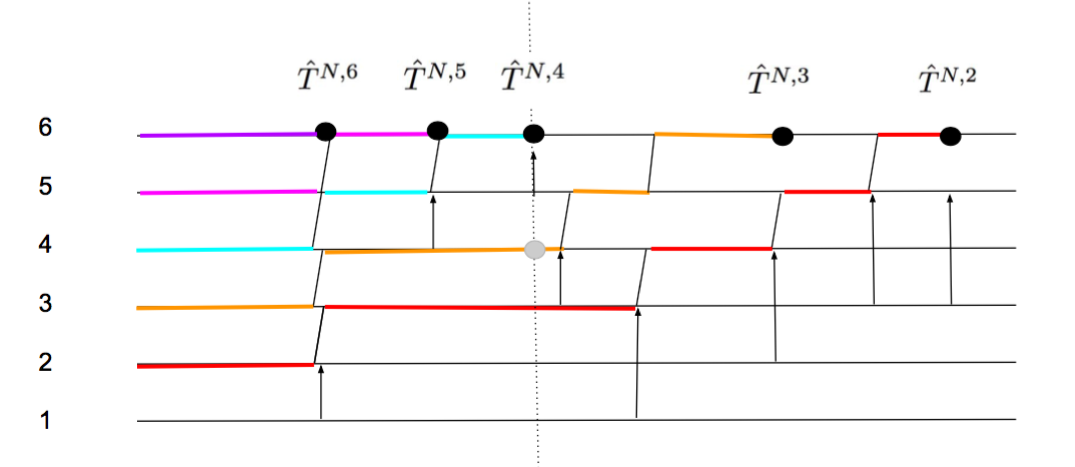}
\caption{In this example: $N=6$, $k=4$ so that  $R^{N,k}=4$. The value of the fixation $L^{k-1}$ (orange line) at time $\hat T^{N,k}$ only depends on the type of arrow 
at the jumps of 
$L^{k}$ (blue curve). $L^{k-1}$ only jumps at the first jump of $L^k$.} 
\label{taut}
\end{figure}
 We will now argue that both (a) and (b) are {\it jointly} independent of $\left(\hat T^{N,i}; k \leq i \leq  N\right)$,
thus showing (\ref{e2}). We start by showing that $R^{N,k}$ is independent of $\left(\hat T^{N,i}; k \leq i \leq  N\right)$.  
Since $L^{k-1}< L^k$, we note that $L^{k-1}$ can only jump if $L^k$ jumps (see again Fig \ref{taut}). More precisely,
if $t$ is a jump time for the fixation line $L^k$ and if
$$
L_{t-}^k \ = \ k+j \quad \mbox{ and } \quad L_{t-}^{k-1}=i   \qquad i<k+j,
$$
then $L_{t}^{k-1}=i+1$ iff the LD event at time $t$ involves a pair $(i_1,i_2)$ with $i_1,i_2\leq i$. This 
happens with probability $\binom{i}{2}/ \binom{k+j}{2}$.

Let us push the previous observation a bit further. For every $j\in[N-k]$, let us consider
\[Z(j) \ = \ L_{\inf\{t>0 \ : \ L^k_t=k+j\}}^{k-1}\]
In words, when $L^k$ reaches level $k+j$, we record the position of the $(k-1)^{st}$
fixation curve. In particular,  $Z(0)=k-1$ and $Z(N+1-k)= R^{N,k}$.
Arguing as before, it is not hard to see that 
$Z$  is a discrete time Markov chain with initial condition $Z(0)=k-1$
and for every $j\in[N-k]$ 
$$
\P\left(Z(j+1) = i+1 \ | \ Z(j)=i \right) \ =  \ 1- \P\left(Z(j+1) = i \ | \ Z(j)=i \right) \ = \  \frac{\binom{i}{2}}{\binom{j+k}{2}},
$$
and by a standard property of Poisson point process, all these events are independent of the jump times of $\left(L^k_t, \ t\in[0,\hat T^{N,k})\right)$. Note that this type of argument is already present in Pfaffelhuber and Wakolbinger \cite{PW}. 
As a consequence, $Z$ is independent of the $\sigma$-field $\mathcal F_k$ generated by $\left(L^k_t, \ t\in[0,\hat T^{N,k})\right)$ and the LD events strictly above $L^k$, that is the LD events $(t,(u,v))$ such that $L^k_t<v$. 
Since $Z(N+1-k)=R^{N,k}$, $R^{N,k}$ is independent of $\mathcal F_k$.

We now claim that $\left(\hat T^{N,i}; k \leq i \leq  N\right)$ is $\mathcal F_k$-measurable.
Indeed for $i\in\{k,\ldots,N\}$, the jumps of $L^i$ are the jumps of $L^k$ in addition to jumps caused by LD events strictly above $L^k$. This shows that the path $\left(L^i_t, \ t\in[0,\hat T^{N,i})\right)$ is $\mathcal F_k$-measurable and so the extinction time $\hat T^{N,i}$ is $\mathcal F_k$-measurable.
Since $R^{N,k}$ is independent of $\mathcal F_k$, $R^{N,k}$ is independent of $\left(\hat T^{N,i}; k \leq i \leq  N\right)$. 

Now let $\mathcal B_k$ denote the $\sigma$-field generated by the LD events after time $\hat T^{N,k}$. Note that both $(\hat T^{N,i}; k \leq i \leq  N)$ and $R^{N,k}$
only depend on the LD events before time $\hat T^{N,k}$. By the strong Markov property, this implies that 
$\mathcal B_k$ is independent of $\left(R^{N,k}, \left(\hat T^{N,i}; k \leq i \leq  N\right) \right)$.
Using the fact that if $A,B,C$ are three r.v.'s with $A\bigCI C$ and $B\bigCI (A,C)$ then $(A,B) \bigCI C$, 
we conclude that $R^{N,k}$ and $\mathcal B_k$ are jointly independent of $\left(\hat T^{N,i}; k \leq i \leq  N\right)$. This ends the proof.

\section{A general urn model}\label{general-urn}

{\bf Some definitions.} In the following $c\in\N$ and corresponds to the number of colors in the urn. Further,
$(p^l)_{l\in\N}$ is a family of probability distributions indexed by $\N$.
 We consider a discrete time increasing Markov process $U_n$ such that conditional $U_n=l$, the jump probability distribution at time $n$
is given by $p^{l}$. We will think of $U_n$
as the size of the urn at time $n$.

For every $V\in \bigcup_{l\in\N} [c]^l$, recall that $\cL(V)$ is the length of the vector $V$. For every finite subset $S \subset \N$, and every $u\in[c]$, we define
$\sigma_S^u(V)$ as the only vector $W$ such that 
\begin{enumerate}
\item for every $j\in S$, $W_j = u$
\item if we remove the lines of $W$ with indices in $S$, the resulting vector is equal to $V$.
\end{enumerate}
Now for every $B \in \N_0^c$, letting $l = \sum_{u=1}^c B_u$, we define $\P_B$ as the uniform probability measure on 
\[\left\{ V \in [c]^l \ : \ \forall u\in[c], \  \sum_{i=1}^l 1_{V_i = u} \ = \ B_{u}  \right\}.\]

\smallskip

{\bf A general urn model.} We now define an urn model, whose state at step $n$ is denoted by $B(n)\in \N_0^c$, where in particular $U_n\ = \ \sum_{u=1}^c B_u(n)$ is the number of balls in the urn. At step $n+1$, conditional on $B(n)$, $B(n+1)$ is generated by the following procedure.
\begin{enumerate}
\item Sample an independent  r.v. $b_n$  according to $p^{U_n}$.
\item 
Conditional on $b_n$, sample a ball uniformly at random from the urn and return it with $b_{n}$ additional
balls of the same color, so that in particular $U_{n+1} = U_n +b_n$.
\end{enumerate}
When $p^l = \delta_1$ for every $l\in\N_0$, this model coincides with the classical Poly\`a urn.


\bigskip

{\bf A particle system.} In the spirit of Donnelly and Kurtz \cite{DK}, we will represent our urn as a particle system. This particle system will be embodied by a process $\left(V(n) ; n\geq0\right)$ valued in $\bigcup_{l\in\N} [c]^l$, where the length $U_n=\cL(V(n))$ of the vector $V(n)$ is the number of particles at time $n$.

At step $n+1$, conditional on $V(n)$, $V(n+1)$ is generated by the following procedure.
\begin{enumerate}
\item 
Sample an independent  r.v. $b_n$  according to $p^{U_n}$. 
\item
Conditional on $b_n$, draw a random finite set $S$ of cardinal $b_n+1$ uniformly in $[U_n + b_n]$.
\item
Conditional on $S$, letting $S' = S \setminus \{\min S\}$ and $u = V_{\min(S)}(n)$, define
\[V(n+1) = \sigma^u_{S'}\left(V(n)\right). \]
\end{enumerate}
The next result provides a particle representation of the urn model defined above.

\begin{prop}\label{prop-inter}
Define the process $(\bar B(n); n\geq 0)$ valued in $\N_0^c$ so that for every color $u\in[c]$, $\bar B_u(n) \ = \ \sum_{i=1}^{\cL(V(n))} 1_{V_i(n)=u}$,
so that $\bar B$ records the number of types in $V$.
 If the law of $V(0)$ is given by $\P_{B(0)}$, then
\begin{enumerate}
\item $\left(\bar B(n); n\geq0\right)$ and $\left(B(n); \ n\geq0\right)$ are identical in law.
\item For every $n\in\N_0$ and ${\mathcal B}\in\N_0^c$, conditional on the event that $\{\bar B(n) \ = \ {\mathcal B}\}$,
the law of $V(n)$ is given by $\P_{{\mathcal B}}$.
\end{enumerate}
\end{prop}
\begin{proof}
The first part of the result can be seen as a direct consequence of Theorem 1.1 in Donnelly and Kurtz \cite{DK}. The second point is apparent in the proof of the same result. For the sake of completeness, we provide an alternative proof of Proposition \ref{prop-inter} based on the intertwining relation of Pitman and Rogers \cite{RP}. 

For any space $E$, define $F(E)$ as the set of bounded functions from $E$ to $\R$. We define the operator $\cV$ from $F(\bigcup_{l\in\N_0} [c]^l)$ to $F(\N_0^c)$ through
the following relation
\[\forall B\in \N_0^c,  \  \cV f(B) \ = \ \E_{B}(f (V) )\]
where the RHS means that the average is taken with respect to the random variable $V$ distributed according to $\P_B$.

According to Pitman and Rogers \cite{RP}, Proposition \ref{prop-inter} will follow if one can show that
for every $B\in \N_0^c$, and every bounded function $h\in F( \bigcup_{l\in \N} [c]^l)$, 
 \begin{equation}\label{eq:inter-2}
 \cV \hat G h(B) \ = \ G \cV h(B)
 \end{equation}
where $G$ is the generator associated to the Poly\`a urn and $\hat G$ is the generator of the particle process defined above.

Let $l \ \equiv \ \sum_{u=1}^c B_u$. We first note that 
\begin{eqnarray*}
\hat G h(V) & = & \left[\sum_{u=1}^c  \sum_{b=1}^l \frac{p^l(b)}{ {l+b \choose b+1}}  \sum_{S \subseteq [l + b] : |S|=b } |\{1\leq x < \min(S) : V(x) = u\}| \  h\left(\sigma_{S}^u (V)\right)\right]  \ - \ h(V) 
\end{eqnarray*}

Let $I$ be the term between brackets. In the following, for every ${\mathcal B}\in \N_0^c$, and every $(u,b)\in[c]\times\N_0$,  ${\theta_{u,b}(B)}$ will refer to 
the vector obtained from $B$
by incrementing the $u$ coordinate of $B$ by $b$ units. We have 
\begin{eqnarray*}
\E_B\left( I \right) & = &  
 \sum_{u=1}^c  \sum_{b=1}^l \frac{p^l(b)}{ {l+b \choose b+1}}  \sum_{S \subseteq [l + b] : |S|=b }  \E_{B}\left( \ |\{1\leq x < \min(S) : V_x = u\}| \  h\left(\sigma_{S}^u (V)\right)   \ \right)  \\
& = &  \sum_{u=1}^c  \sum_{b=1}^l \frac{p^l(b)}{ {l+b \choose b+1}}  \sum_{S \subseteq [l + b] : |S|=b }  \E_{B}\left( \ |\{1\leq x < \min(S) : \sigma_{S}^u (V)_x = u\}| \  h\left(\sigma_{S}^u (V)\right)   \ \right)   \\
& = & 
\sum_{u=1}^c  \sum_{b=1}^l \frac{p^l(b)}{ {l+b \choose b+1}}  \sum_{S \subseteq [l + b] : |S|=b }  \E_{\theta_{u,b}(B)}\left( \ |\{1\leq x < \min(S) :  V_x= u\}| \  h\left(V \right) \ | \ \forall j\in S, \ V_j \ = \ u \right)   \\
& = &
\sum_{u=1}^c  \sum_{b=1}^l \frac{p^l(b)}{ {l+b \choose b+1}}  \sum_{S \subseteq [l + b] : |S|=b }  \E_{\theta_{u,b}(B)}\left( \ |\{1\leq x < \min(S) :  V_x= u\}| 1_{\forall j\in S, \ V_j \ = \ u } \  h\left(V \right)  \right)  \\
& \times & 1 / \P_{\theta_{u,b}(B)}( \ \forall j\in S, \ V_j \ = \ u) 
\end{eqnarray*}
where in the second identity we used the fact that $\sigma_S^u(V)_x=V_x$ for $x<\min(S)$, and the third identity follows from the fact
 that for every bounded function $m$:
\[ \E_{B} \left( m(\sigma_S^{u}(V)) \  \right)  \ = \ \E_{\theta_{u,|S|}(B)}\left(  \ m(V) \ | \   \ \forall j\in S, V_j = u \right).\]
On the other hand, if $B$ is such that $B_u>b$ then
\[\sum_{S \subseteq [l + b] : |S|=b } {1}_{\forall j\in S, \ V_j \ = \ u } \   |\{1\leq x < \min(S) :  V_x= u\}|   \ = \ {B_u \choose b+1}, \ \ \P_{B}  \ \mbox{a.s.}\]
whereas
\[ \P_{B}( \ \forall j\in S, \ V_j \ = \ u) \ = \ \frac{B_u \cdots (B_u - (|S|-1))}{l \cdots (l-(|S|-1))}, \ \ \mbox{where $l \  = \ \sum_{u=1}^c B_u$,} \]
and thus, applying the two previous formulas after replacing $B$ by $\theta_{u,b}(B)$
\begin{eqnarray*}
\E_B\left( I \right) & =  & \sum_{u=1}^c  \sum_{b=1}^l \frac{p^l(b)}{ {l+b \choose b+1}} \frac{(l+b)\cdots(l+1)}{(B_u+b)\cdots(B_u+1)} {B_u+b \choose b+1}  \E_{\theta_{u,b}(B)}(h(V))  \\
& = & \sum_{u=1}^c \frac{B_u}{l}  \sum_{b=1}^l  p^l(b)   \E_{\theta_{u,b}}(h(V)).  
\end{eqnarray*}
Further, since 
\begin{eqnarray*}
\E_B\left( h(V) \right) 
& = & \sum_{u=1}^c \frac{B_u}{l}  \sum_{b=1}^l  p^l(b)   \E_{B}(h(V)),  
\end{eqnarray*}
we get  that
\begin{eqnarray*}
\cV \hat G h(B)  & = & \sum_{u=1}^c \frac{B_u}{l}  \sum_{b=1}^l  p^l(b)  \left( \E_{\theta_{u,b}(B)}(h(V)) \   - \ \E_B(h(V)) \right) \\
		   & = &  G \cV h(B).
\end{eqnarray*}
As already mentioned, this completes the proof of Proposition \ref{prop-inter} by an application of the intertwining theorem of Rogers and Pitman \cite{RP}.
\end{proof}

\section{Proof of Theorem \ref{teo-block-type2} and Proposition \ref{extinction-polya}}\label{sect:proof-bt}

\subsection{Definition of $\hPi^k$}
We start by defining a quantity $\hPi^k$ in the LD process, which will be the analog of the quantity $\hPi^{\infty,k}$
introduced in Theorem \ref{teo-block-type2}.
Consider the embedded Markov Chain associated to the Markov process 
\[ \left( \xi_1(t),\ldots, \xi_{L_t^k(t)-1}(t) \right), \]
where $\xi(t)$ is the LD process at time $t$. More specifically, $\xi_i(t)$ is the allele carried by level $i$ at time $t$, assuming that at time $t=0$, types are  i.i.d. uniform in $[0,1]$.
From the definition of $L^k$, there are exactly $k-1$ different different values (alleles) along the coordinates of the chain at any time $n$.
As a consequence, this chain can be mapped to a chain $(V^k(n); n\geq 0)$ valued in $\{1,\ldots,k-1\}$, by replacing each type by its rank in the ordered statistics of alleles (i.e., the smallest $\xi_i(t)$ present in the sample is assigned type $1$ etc.)   
It is not hard to check that $(V^k(n); n\geq 0)$ is identical in law to the particle process described in Section \ref{general-urn} with 
$p^{U_n}$ (the jump law of the vector size at level $U_n$) being given by the law of $Y^{U_n}$ as defined in (\ref{eq:yn}). Indeed, conditional on $U_n=l$, where $U_n=\cL(V^k(n))$ is the length of $V^k(n)$, the next LD-event with an effect on the look-down process up to level $l$ is the first LD-event $(t,A)$ such that $\#A\cap [l]\geq 2$ and then the level that gives birth is $\min A$ and the number of newborn levels is $\#A -1$. As a consequence, 
$$
\P(U_{n+1}-U_n=b-1) = \frac{L(\{A\subset [l]: \#A= b\})}{L(\{A\subset [l]: \#A\geq 2\})}\qquad b\in\{2,\ldots, l\},
$$
and a quick calculation based on the definition of $L$ given by \eqref{eqn:def-LD} shows that this is exactly the distribution of $Y^l$ given in (\ref{eq:yn}). Therefore, we have proved that Rule (1) as specified in the definition of the particle system in Section \ref{general-urn}, is satisfied. Now conditional on $U_{n+1}-U_n=b-1$, the Poissonian construction of the LD-event $(t,A)$ implies that $A$ is uniformly chosen among the subsets of $[l]$ with cardinal $b$, so that Rule (2) is also satisfied.  Last, Rule (3) describing the sharing of new particles amongst old particles, was exactly designed to fit the reordering prescribed by the look-down process at a LD-event.  

Further, the initial distribution of the initial configuration $V^k(0)$ is given by $\P_{\underbrace{(1,\ldots,1)}_{k-1}}$.
Define
\[B_u^k(n) \ = \ \sum_{i=1}^{\cL(V^k(n))} 1_{V_i^k(n)=u}\qquad u\in[k-1]. \]
Then according to Proposition \ref{prop-inter}, $\left(B^k(n); n\geq0\right)$  
defines a  $\Lambda$-urn with initial condition $\underbrace{(1,\ldots,1)}_{k-1}$. 

By Proposition \ref{ref:conv}, there exists a mass partition which is obtained as the limit of the sequence of mass partitions 
$\left( B^k_u(n) / \sum_{u=1}^{k-1} B^k_{u}(n) ; u\in \{1,\ldots,k-1\} \right)$. 
This mass partition will be denoted by $\hPi^k$. Note that by definition  $\hPi^k$ is distributed as 
$\D_{k-1}^\Lambda$. (When $\Lambda=\Lambda(\{0\})\,\delta_0$ this is the standard Dirichlet mass partition.) We now show that $\hPi^k$ corresponds to the quantity $\hPi^{\infty,k}$
introduced in Theorem \ref{teo-block-type2}. 

\begin{lem}[Large population limit]\label{conv} 
For every $k\geq 2$, $\hPi^{N,k}  \rightarrow \ \hPi^k$ a.s. as $N\to\infty$. 
\end{lem}
\begin{proof}
Recall that $\hPi^{N,k}$ encapsulate the frequencies of the $k-1$ alleles strictly below the fixation line starting at level $k$ at $\hat T^{N,k}$, i.e., when the fixation line goes strictly above level $N$. This can be reformulated by saying that $\hPi^{N,k}$ is the vector of allele frequencies in the $\Lambda$-urn (as defined in the previous paragraph) when the number of balls exceeds $N$ for the first time. Analogously, $ \hPi^k$
was defined as the vector of asymptotic frequencies in the urn. The result then follows by a direct application of Proposition \ref{ref:conv}.
\end{proof}

\begin{rmk}
Note that $\hPi^k$ is potentially degenerate.
\end{rmk}

\subsection{LD Configuration at time $\hat T^k$. Proof of Theorem \ref{teo-block-type2}(1)}\label{sect:sit-ld}

As a simple extension of Proposition \ref{prop-inter}(2) (by going to the limit), it is not hard to see that conditional on $\hPi^k$, the LD process at time $\hat T^k$ is obtained by assigning alleles independently for every level
according to the distribution $\hPi^k$. Next,
we note that the definition of the fixation line $L^{k-1}$ readily implies that
\[\forall t\geq0, \   \ \   L^{k-1}_{t} \ = \ \inf\{l\in\N \ : \ \# \{\mbox{types carried by levels in $[l]$ at time $t$}\} = k-1 \}\]
Let us now assume that $\hPi^k$ is degenerate. The two previous observations would  imply that  $L^{k-1}_{\hat T^k}=+\infty$. Since we assumed that there is no simultaneous extinctions, this yields a contradiction.
This  completes the proof of the first part of Theorem \ref{teo-block-type2}.  

Before proceeding with the second part of the proof,
we mention the following lemma which is a direct consequence of the previous observations.

\begin{lem}\label{lem:1}
Under the assumptions of Theorem \ref{teo-block-type2},  $\hPi^k$ is non-degenerate and conditional on $\hPi^k$, the mass partition
\[B= \left( \sum_{i=1}^{L^{k-1}_{\hat T^k}-1} \ 1_{V_i =  u}; \ u\in[k-2] \right) \]
is identical in law to the  random coupon collection
associated to  $\hPi^k$.

Also, conditional on $B$, the particle configuration
$(\xi_1,\ldots,\xi_{L^{k-1}_{\hat T^k}-1})$ is distributed as $\P_{B}$.  
\end{lem}

\subsection{Markov property and transition probabilities}

\begin{cor}\label{cor:markov}
For every $K\in\N$, $(\hPi^{K-k}; k\in\{0,\ldots, K-2\})$ is a Markov process.
\end{cor}
\begin{proof}
The result is obviously true for finite population; at least if we look at the original  $(N,\Lambda)$-FV process (not the one defined from the LD process).
The proof is completed by going to the limit.
\end{proof}

In order to describe the transition probabilities of the Markov process described in Corollary \ref{cor:markov}, we fix $k\geq 3$ and we describe the distribution of $\hPi^{k-1}$ given $\hPi^{k}$.
Define $\Net^k$ to be the number of jumps of the fixation line $L^{k-1}$ right before the explosion time $\hat T^k$ (see Fig \ref{taut}), and set 
\[  \tilde V^{k-1}(n) \ = \ V^{k-1}(n + \Net^k)\qquad n\in\N_0.\]
Intuitively, $\tilde V^{k-1}$ corresponds to the description of the urn with index $k-1$ just after the explosion of the urn with index $k$. 
By definition, $\hPi^{k-1}$ is obtained by computing the asymptotic frequencies of the $k-2$ types present in the shifted process $\tilde V^{k-1}$.
Analogously to the proof of Theorem \ref{teo:1} (see again Fig \ref{taut}), we first note that $\tilde V^{k-1}$ is a function of 
\begin{enumerate}
\item[(a)] the types carried by the levels strictly below level $L_{\hat T^k}^{k-1}$ at time $\hat T^k$  (in the original LD process).
\item[(b)] the LD events after time $\hat T^k$.
\end{enumerate}
From Lemma \ref{lem:1}, it follows that conditioned on $\hPi^k$,
the process $\tilde B^{k-1}(0) := B^{k-1}(\Net^k)$ is identical in law to the  random coupon collection
associated to the non-degenerate $\hPi^k$, and further, conditional on $\tilde B^{k-1}(0)$, the particle configuration
$\tilde V^{k-1}(0)$ is distributed as $\P_{\tilde B^{k-1}(0)}$.

Using again Proposition \ref{prop-inter} and the fact that (b) above
is independent of $\hPi^k$,
it follows that conditioned on $\hPi^{k}$, $\hPi^{k-1}$ is identical in law to a $\Lambda$-urn whose initial condition can be described in terms of a coupon collection
generated from $\hPi^{k}$. This is exactly the transition mechanism described in Theorem \ref{teo-block-type2}.

\begin{proof}[Proof of Theorem \ref{teo-block-type2}]
The result follows by combining Lemma \ref{conv}, Corollary \ref{cor:markov} and the correlation between $\hPi^k$ and $\hPi^{k-1}$ exposed above. 
\end{proof}

\begin{proof}[Proof of Proposition \ref{extinction-polya}]

\underline{(ii) $\Longrightarrow$ (i).} If ${\mathcal D}_{k}^\Lambda$ has exactly $k$ coordinates a.s., then Theorem \ref{teo-block-type2} implies that then at extinction time $\hat T^{k+1}$,
there is exactly $k$ types remaining, (and further,  those types have a strictly positive frequency). Thus we must have $\hat T^k>\hat T^{k+1}$.

\underline{(i) $\Longrightarrow$ (ii).} This was already the object of Section \ref{sect:sit-ld}.
\end{proof}

\paragraph{\bf Acknowledgments.} The authors thank the {\em Center for Interdisciplinary Research in Biology} (Coll\`ege de France) for funding.

\bibliographystyle{abbrv}
\bibliography{refs}

\end{document}